\documentclass[a4paper,oneside,11pt]{article}

\usepackage{amsmath,amsfonts,amscd,amssymb,dsfont}
\usepackage{longtable,geometry}
\usepackage[english]{babel}
\usepackage[utf8]{inputenc}
\usepackage[active]{srcltx}
\usepackage[T1]{fontenc}
\usepackage{graphicx}
\usepackage{pstricks}
\usepackage{bbm}
\usepackage{enumitem}

\usepackage{color}
\usepackage{MnSymbol}
\usepackage{stmaryrd}
\usepackage{nicefrac}
\usepackage{calrsfs}
\newtheorem{theorem}{Theorem}

\newtheorem{corollary}[theorem]{Corollary}
\newtheorem{lemma}[theorem]{Lemma}
\newtheorem{proposition}[theorem]{Proposition}
\newtheorem{definition}[theorem]{Definition}

\newtheorem{remark}[theorem]{Remark}

\newcommand{\be}[1]{\begin{equation}\label{#1}}
\newcommand{\ee}{\end{equation}}
\numberwithin{equation}{section}

\newcommand{\ba}[1]{\begin{align}\label{#1}}
\newcommand{\ea}{\end{align}}
\numberwithin{equation}{section}

\newcommand{\ben}{\begin{equation*}}
\newcommand{\een}{\end{equation*}}
\numberwithin{equation}{section}

\newenvironment{proof}[1][\relax]
  {\paragraph{Proof\ifx#1\relax\else~of #1\fi}}%
  {~\hfill$\square$\par\bigskip}

\newcommand{\calE}{\mathcal{E}}
\newcommand{\calF}{\mathcal{F}}
\newcommand{\calG}{\mathcal{G}}

\newcommand{\calJ}{\mathcal{J}}

\newcommand{\calS}{\mathcal{S}}
\newcommand{\calT}{\mathcal{T}}

\newcommand{\calX}{\mathcal{X}}

\newcommand{\bbE}{\mathbb{E}}

\newcommand{\bbP}{\mathbb{P}}

\newcommand{\bbR}{\mathbb{R}}

\newcommand{\bbZ}{\mathbb{Z}}

\newcommand{\Pim}{\vec\Pi_{\rm max}^{(x)}}

\newcommand{\rk}[1]{\bgroup\color{red}%
  \par\medskip\hrule\smallskip%
  \noindent\textbf{#1}%
  \par\smallskip\hrule\medskip\egroup}

\title{On the number of maximal paths in directed last-passage percolation}
\author{Hugo Duminil-Copin, Harry Kesten, Fedor Nazarov, \\
Yuval Peres and Vladas Sidoravicius}
\date{\today}

\begin{document}

\maketitle

\begin{abstract}
We show that the number of maximal paths in directed last-passage percolation on the hypercubic lattice ${\mathbb Z}^d$ $(d\geq2)$ in which weights take finitely many values is typically exponentially large.
\end{abstract}
\section{Introduction}
Consider a family of random variables, called weights, indexed by sites of $\mathbb Z^d$ with $d\ge2$. Associate with each path a weight equal to the sum of the weights of its sites. 
Directed last-passage percolation (DLPP) is a variant of first passage percolation studying directed paths of maximal weight. We refer to \cite{Mar06} for a survey on this model.

Formally, let $\mathbb Z^d$ be the hypercubic lattice of points with integer coordinates. For $1\le i\le d$, let $e_i$ be the site of $\bbZ^d$ with $i$-th coordinate equal to $1$, and other coordinates equal to 0. 

A {\em directed path} $\pi$ connecting $x$ to $y$ is a sequence of sites $\pi=\{x=\pi_1,\dots,\pi_k=y\}$ such that $\pi_{i+1}-\pi_i\in\{e_1,\dots,e_d\}$ for each $1\le i<k$. The integer $k$ is called the {\em length} of $\pi$. 
Let $(\omega_x:x\in\mathbb Z^d)$ be a collection of random variables, called {\em weights}, indexed by sites of $\bbZ^d$.
Given a path $\pi$, the quantity $W_\pi=\sum_{x\in\pi}\omega_x$ is called the {\em weight} of $\pi$. 
For $n\ge1$, a path such that $W_\pi$ is maximized among directed paths of length $n$ starting from 0 is called a {\em maximal path of length $n$}. 
We define $\vec\Pi^{(n)}_{\rm max}$ to be the set of maximal paths of length $n$. 
The main result of this paper is the following.
\begin{theorem}\label{thm:main2}
If $(\omega_x:x\in\mathbb Z^d)$ are iid random variables taking values in a finite set $\Theta\subset\bbR$, then there exists $\delta>0$ such that for every $n$ large enough,
$$\bbP(\,  |\vec\Pi_{\rm max}^{(n)}|   \le   2^{\delta n})\le \exp(-\delta n).$$
\end{theorem}

Note that the constant $\delta>0$ depends on the distribution of $\omega_0$: we do not claim any uniformity with respect to this distribution. 

Counting the number of maximal paths (for a prescribed random environment) naturally arises in the study of directed growth models and directed polymers in
random environment, see e.g.~\cite{FukYos12}. We refer to \cite{Mar06} and references therein for additional details. In many cases, the environment is constituted of independent Bernoulli random variables of parameter $p$, where $p$ exceeds the critical probability $p_c$ of oriented percolation. We refer to \cite{GarGouMar13,Lac12} and references therein for the $p>p_c$ case and to  \cite{ComPopVac08,KesSid10} for results dealing with the general case (see also the recent paper \cite{Shu}).

The previous result extends to point-to-point maximal paths. More precisely, for $x\in\bbZ^d$, a path such that $W_\pi$ is maximized among directed paths from $0$ to $x$ is called a {\em maximal path from 0 to $x$}. 
We define $\vec\Pi^{(x)}_{\rm max}$ to be the set of maximal paths from $0$ to $x$.

Denote the $\ell^1$ norm of $x=(x_1,\dots,x_d)$ by $\|x\|:=|x_1|+\dots+|x_d|$. 
For $\beta>0$, introduce the {\em cone} $C_\beta\subset \bbZ_+^d$ of sites $x=(x_1,\dots,x_d)$ with $0\le x_i\le (1-\beta)\|x\|$ for every $1\le i\le d$. 

\begin{theorem}\label{thm:main}
If $(\omega_x:x\in\mathbb Z^d)$ are iid random variables taking values in a finite set $\Theta\subset\bbR$, then for every $\beta>0$, there exists $\delta=\delta(\beta)>0$ so that for every $x\in C_\beta$ with $\|x\|$ large enough,
$$\bbP(\,  |\vec\Pi_{\rm max}^{(x)}|   \le   2^{\delta \|x\|})\le \exp(-\delta \|x\|).$$
\end{theorem}

 Let us discuss briefly the strategy of the proof. A {\em turn} of  $\pi$ is a site $x=\pi_i$ (with $i$ strictly between 1 and the length of $\pi$) such that $\pi_{i+1}-\pi_i\ne \pi_i-\pi_{i-1}$. A turn $x$ is called a {\em bifurcation} of $\pi$ if $\omega_x=\omega_{x^*}$, where $x^*=x^*(\pi)=\pi_{i-1}+\pi_{i+1}-\pi_i$ (see Fig.~\ref{fig:1}). The proof of Theorem~\ref{thm:main} is based on two successive steps:

\noindent 1.~Show that {\em every} maximal path has a positive density of turns,\\
\noindent 2.~Show that {\em some} maximal path has a positive density of bifurcations.\\
The second step immediately implies the results: one may create an exponentially large number of maximal paths by modifying locally the maximal path provided by Step 2 near the bifurcations (maximal paths may go through $x^*$ instead of $x$ for each bifurcation). 

The first step is very easy to accomplish. The new contribution of this article lies in the proof of the second step.
The proof is based on local modifications of maximal paths near their turns. These local transformations enable us to bound the probability of only having maximal paths with few bifurcations using Lemma~\ref{principle}.

%
%
%
%
%


We use independence only in the first step. We did not attempt to prove this statement in a more general context, but it seems likely that such an extension exists. 

It also seems very likely that the theorem can be generalized to the non-directed case. In this case, the maximization problem is done on self-avoiding paths of length $n$. The techniques applied here should extend, even though one should be careful to modify them accordingly, which would require a certain amount of work. Similarly, the argument may extend to countable sets $\Theta$, and even to unbounded random variables under some mild moment assumptions. In order to highlight as much as possible the ideas in this paper, we chose not to produce too long a proof, and in particular not to discuss these types of problems here.

\paragraph{Notation} Below, we will be interested only in large values of $n$. For this reason, we will often make the implicit rounding operation consisting in taking the integer part of real numbers. For instance, $\binom{n}{\delta n}$ will mean $\binom{n}{\lfloor \delta n\rfloor}$.

We fix iid weights $(\omega_x:x\in\bbZ^d)$ taking values in a finite set $\Theta$. Without loss of generality, we assume that $\min\Theta=0$ and $\max \Theta=1$ (we will only consider non-degenerate random variables since otherwise the result is trivial). We also set 
$$p:=\min\{\bbP(\omega_0=0),\bbP(\omega_0=1)\}>0.$$  

All the constants mentioned in the proofs depend on the distribution, but we will not refer to it anymore. We will also work with point-to-point maximal paths, and will therefore refer to maximal paths instead of maximal paths from 0 to $x$ when the context is clear.

We will often work with two configurations $\omega$ and $\omega'$. For convenience, we will consistently use $W_\pi$ and $W'_\pi$ for the respective weights of $\pi$ in $\omega$ and $\omega'$. Also, we denote  the weight of a maximal path from $0$ to $x$ in $\omega$ by $W_{\rm max}^{(x)}$ (we will never use the notation for $\omega'$, so that no confusion will be possible).

Let $\mathbf{turn}_\pi\subset\mathbb Z^d$ be the set of turns of $\pi$. 
For $x\in\mathbf{turn}_\pi$, define
$$\mathbf{shield}_\pi(x):=\big\{y\notin\pi:y-x^*\in\{\pm e_1,\dots,\pm e_d\}\big\},$$
which is the set of neighbors of $x^*$ not in $\pi$, see Fig.~\ref{fig:1}.

Last, for two sites $x$ and $y$ of a directed path $\pi$, let $\pi[x,y]$ be the portion of $\pi$ from $x$ to $y$. Similarly, we define $\pi(x,y)=\pi[x,y]\setminus\{x,y\}$.

\paragraph{Organization of the paper}
The next section contains preliminaries. In particular, it studies the number of turns on a typical maximal path. It also provides an explanation of why Theorem~\ref{thm:main2} follows from Theorem~\ref{thm:main}. Finally, it introduces a multi-valued map principle (Lemma~\ref{principle}) which will be used extensively in the next sections. Section~3 is devoted to the proof of Theorem~\ref{thm:main}. 
\begin{figure}\label{fig:1}
\center{\includegraphics[width=0.50\textwidth]{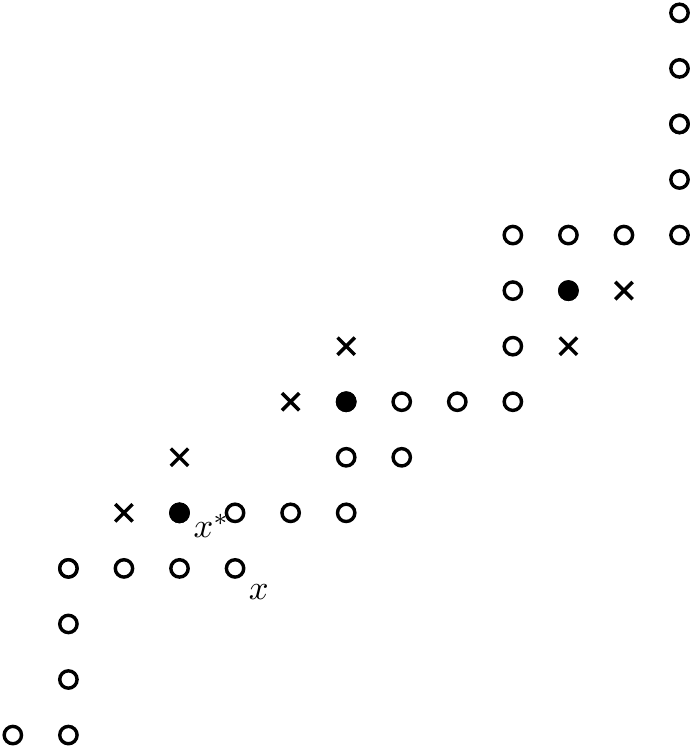}}
\caption{The path $\pi$ is depicted by white bullets. The black sites are sites of the form $x^*$. For three of them, we depicted $\mathbf{shield}_\pi(x)$ with crosses (note that it may be below or above the curve in two dimensions). 
}
\end{figure}

\section{Preliminaries}

\subsection{Maximal paths have many turns}

We start by a simple proposition stating that the maximal weight is exceeding the average weight.

\begin{proposition}\label{lem:bounded above}
For every $\beta>0$, there exists $\mu=\mu(\beta)>\bbE[\omega_0]$ and $c=c(\mu,\beta)>0$ such that for every $x\in C_\beta$ with $\|x\|$ large enough,
\begin{equation}\label{eq:limit exponential}\bbP\big(W_{\rm max}^{(x)}< \mu \|x\|\big)\le \exp(-c\|x\|).\end{equation}
\end{proposition}

\begin{proof}
Pick $x\in C_\beta$ with $\|x\|$ large enough. Fix a directed path from $0$ to $x$ having $k\ge \beta \|x\|$ turns. Consider a subset $S$ of ${\bf turn}_\pi$ composed of at least $\lfloor k/2\rfloor$ turns at $\|\cdot\|$-distance at least two of each other. Since for every element $y\in S$, one may choose to go through $y^*$ instead of $y$ (regardless of the choices made before or after, thanks to the fact that turns are at a distance at least two of each other), we find that
$$W^{(x)}_{\rm max}\ge \sum_{y\in S} \max\{\omega_y,\omega_{y^*}\}+\sum_{y\in \pi\setminus S}\omega_y.$$
Note that the variables on the right-hand side are independent. Since there are at least $\lfloor k/2\rfloor$ elements in $S$ associated with random variables satisfying $$\bbE[\max\{\omega_y,\omega_{y^*}\}]>\bbE[\omega_y],$$ the claim follows directly from large deviations theory for independent bounded random variables. \end{proof}
 
 \begin{proposition}\label{prop:many turns}
For every $\mu>\mathbb E[\omega_0]$, there exists $\kappa>0$ such that for every $x\in\bbZ_+^d$ with $\|x\|$ large enough,
 \begin{align}\label{eq:many turns}
 \bbP(W_{\rm max}^{(x)}\ge\mu\|x\|\text{ and }\exists \pi\in \vec\Pi^{(x)}_{\rm max}:|\mathbf{turn}_\pi|< \kappa \|x\|)\le \exp(-\kappa \|x\|).
\end{align}
\end{proposition}

\begin{proof}
Large deviation estimates for bounded iid random variables imply that there exists $c'>0$ such that for every fixed path $\pi$,
\begin{equation}\label{eq:cdd}\bbP(W_\pi\ge \mu \|x\|)\le \exp(-c' \|x\|)
.\end{equation}
Since there are less than $\binom {\|x\|}{\kappa \|x\|}(d-1)^{\,\kappa \|x\|}$ paths with less than $\kappa \|x\|$ turns, the union bound implies that the probability that there exists a path with less than $\kappa\|x\|$ turns with weight larger than $\mu\|x\|$ is  exponentially small. \end{proof}

As a consequence of the two previous propositions, we obtain the following corollary.
\begin{corollary}\label{cor:many turns}
For every $\beta>0$, there exists $\kappa>0$ such that 
for every $x\in C_\beta$  with $\|x\|$ large enough,
 \begin{align}
 \bbP(\exists \pi\in \vec\Pi^{(x)}_{\rm max}:|\mathbf{turn}_\pi|< \kappa \|x\|)\le \exp(-\kappa \|x\|).
\end{align}

\end{corollary}

\subsection{From Theorem~\ref{thm:main} to Theorem~\ref{thm:main2}}

For every integer $n$ large enough, choose $x_n\in C_{1/3}$ with $\|x_n\|=n-1$. Since paths from $0$ to $x_n$ are of length $n$, we deduce that 
\begin{equation}\bbP(W_{\rm max}^{(n)}<\mu n)\le \bbP(W_{\rm max}^{(x_n)}<\mu n)\le \exp(-cn),\label{eq:ji}\end{equation}
where $\mu>\bbE[\omega_0]$ and $c>0$ are given by Proposition~\ref{lem:bounded above} applied to $\beta=1/3$.
Fix $\kappa=\kappa(\mu)>0$ so small that Proposition~\ref{prop:many turns} implies that for every $x\in\bbZ_+^d$ with $\|x\|$ large enough,
$$ \bbP(W_{\rm max}^{(x)}\ge\mu\|x\|\text{ and }\exists \pi\in \vec\Pi^{(x)}_{\rm max}:|\mathbf{turn}_\pi|< \kappa \|x\|)\le \exp(-\kappa \|x\|).$$
Fix $\beta=\beta(\kappa)>0$ so small that any oriented path from $0$ to $x\notin C_{\beta}$ contains fewer than $\kappa\|x\|$ turns. We deduce from the previous inequality that for every $x\notin C_{\beta}$,
\begin{equation} \bbP(W_{\rm max}^{(x)}\ge\mu\|x\|)\le \exp(-\kappa \|x\|).\label{eq:jji}\end{equation}
Now, let $\delta=\delta(\beta)>0$ such that Theorem~\ref{thm:main} holds true.
We find that
\begin{align*}\bbP(|\Pi_{\rm max}^{(n)}|\le 2^{\delta n})&\le \bbP(W_{\rm max}^{(n)}<\mu n)+\sum_{\substack{x\notin C_{\beta}\\ \|x\|=n-1}}\bbP(W_{\rm max}^{(x)}\ge \mu n)+\sum_{\substack{x\in C_{\beta}\\ \|x\|=n-1}}\bbP(|\Pi_{\rm max}^{(x)}|\le 2^{\delta n})\\
&\le \exp(-cn)+O(n^{d-1})\exp(-\kappa n)+O(n^{d-1})\exp(-\delta n).\end{align*}
The claim follows readily.
%
%
%
%

\subsection{A multivalued map principle}

We will bound the probability of events using the following multi-valued map principle. For a set $\calF$, let $\mathfrak P(\calF)$ be the power set of $\calF$. 
\begin{lemma}[multi-valued map principle]\label{principle}
For every $\varepsilon>0$ and every two events $\calE$ and $\calE'$, assume that there exist a set $\calF$ and two maps 
\begin{align*}\calS&:\calE\longrightarrow \mathfrak P(\calF)\\
T&:\{(\omega,F)\in\calE\times\calF\text{ such that }F\in \calS(\omega)\}\longrightarrow \calE'\end{align*}
satisfying that for every $\omega'\in\calE'$ and $F\in\calF$,
\begin{align}\label{eq:condition}\bbP(\omega')\ge \varepsilon\, \bbP(\omega\in\calE\text{ such that }F\in\calS(\omega)\text{ and }\omega'=T(\omega,F)).\end{align}
Then, 
\begin{equation}\label{eq:principle}\bbP(\calE)~\le~ \frac{\displaystyle\max_{\omega'\in\calE'}|\calT(\omega')|}{\varepsilon\cdot\displaystyle\min_{\omega\in \calE}|\calS(\omega)|}~\bbP(\calE'),\end{equation}
where $$\calT(\omega'):=\{F\in\calF:\exists \omega\in\calE\text{ such that }F\in\calS(\omega)\text{ and }\omega'=T(\omega,F)\}.$$
\end{lemma}

\begin{proof}Simply write
\begin{align*}\bbP(\calE)=\sum_{\omega\in\calE}\bbP(\omega)&= \sum_{\omega\in\calE}\sum_{F\in \calS(\omega)}\frac{\bbP(\omega)}{|\calS(\omega)|}\\
&\le \frac1{\displaystyle\min_{\omega\in \calE}|\calS(\omega)|}\sum_{\omega\in\calE}\sum_{F\in \calS(\omega)}\bbP(\omega)\\
&= \frac1{\displaystyle\min_{\omega\in \calE}|\calS(\omega)|}\sum_{\omega'\in\calE'}\sum_{F\in \calT(\omega')} \bbP(\omega\in\calE:F\in\calS(\omega)\text{ and }\omega'=T(\omega,F))\\
&\le \frac1{\displaystyle\varepsilon\cdot\min_{\omega\in \calE}|\calS(\omega)|}\sum_{\omega'\in\calE'}\sum_{F\in \calT(\omega')} \bbP(\omega')\\
&\le \frac{\displaystyle\max_{\omega'\in\calE'}|\calT(\omega')|}{\displaystyle\varepsilon\cdot\min_{\omega\in \calE}|\calS(\omega)|} \ \bbP(\calE')
\end{align*}(the second inequality is due to \eqref{eq:condition}), which is the claim.
\end{proof}

For $x\in\bbZ^d_+$, we always consider $\calS(\omega)$ to be a collection of pairs $(\pi,S)$ with $\pi$ a maximal path from 0 to $x$ of $\omega$ and $S\subset \pi$. Except for Lemma~\ref{prop:intermediate}, we will have that $S\subset{\bf turn}_\pi$. From now on, we restrict our attention to such maps and write $T(\omega,\pi,S)$ instead of $T(\omega,(\pi,S))$. In what follows, the set $S$ should be understood as places near which the configuration is modified in such a way that sites of $S$ become bifurcations of $\pi$ in $T(\omega,\pi,S)$. 

The following notion of local transformation will be crucial in the future. Roughly speaking, locality means that maximal paths in $\omega'$ are easy to identify: they are maximal paths in $\omega$, except that they may go through $y^*$ instead of $y$ for every new bifurcation $y$ of the path $\pi$, and that these new bifurcations are themselves only localized at elements of $S$. 

\begin{definition}\label{def:1}Call $T$ {\em local}  if for every $\omega'=T(\omega,\pi,S)$,
\begin{enumerate}[noitemsep,nolistsep]
\item $\pi$ is a maximal path of $\omega'$ and the set of bifurcations of $\pi$ in $\omega'$ is the union of $S$ and a subset of the set of bifurcations of $\pi$ in $\omega$.  
\item For any maximal path $\pi'$ of $\omega'$ satisfying that for every $z_0,z_1\in\pi$ such that $\pi'[z_0,z_1]\cap\pi=\{z_0,z_1\}$, one of the following two conditions holds:

- the sums of weights in $\omega$ of $\pi'(z_0,z_1)$ and $\pi(z_0,z_1)$ are equal, 

- there exists $y\in S$ such that $\pi'(z_0,z_1)=\{y^*\}$.
\end{enumerate}
\end{definition}

The part $\pi'(z_0,z_1)$ of the path $\pi'$ should be understood as an excursion away from $\pi$ (see Fig.~\ref{fig:5}). The second condition yields that every such excursion is either reduced to a single site at a bifurcation, or is already part of a maximal path in $\omega$.

\begin{figure}
\center{\includegraphics[width=0.50\textwidth]{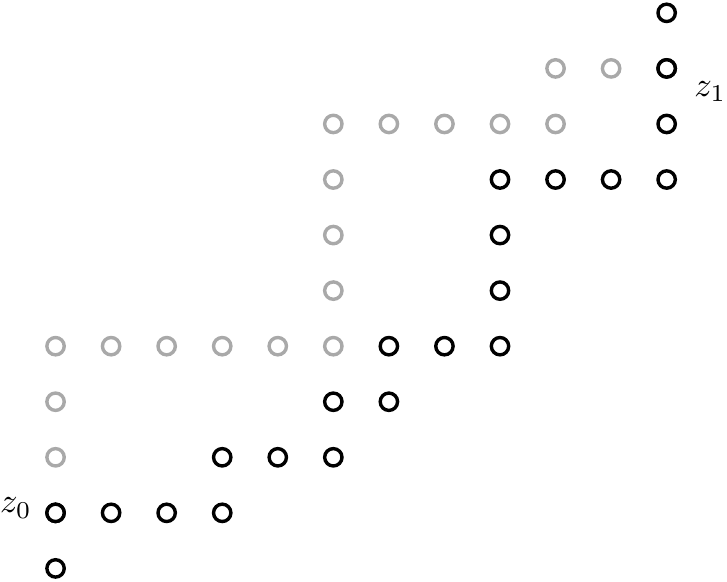}}
\caption{\label{fig:5}\label{fig:8}We depicted $\pi'(z_0,z_1)$ in gray.}
\end{figure}

 The notion of locality will be important when bounding the cardinality of $\calT(\omega')$ thanks to the following lemma.
\begin{lemma}\label{lem:local}
Let $\delta>0$ and $x\in\bbZ_+^d$. Assume that for every $\omega\in \calE$, $|\Pim(\omega)|\le 2^{\delta \|x\|}$ and $|S|\le \delta \|x\|$ for every $(\pi,S)\in\calS(\omega)$. 
If $T$ is local, then for every $\omega'\in\calE'$,
$$|\calT(\omega')|\le 2^{5\delta \|x\|}.$$ 
\end{lemma}

\begin{proof}
Set $n:=\|x\|$. Fix $\omega'\in\calE'$. 
First, let us bound the number of paths $\pi$ such that there exist $\omega\in\calE$ and $S$ with $|S|\le \delta n$ satisfying that $(\pi,S)\in\calS(\omega)$ and $\omega'=T(\omega,\pi,S)$. By Item 1 of locality, we simply need to bound the number of maximal paths in $\omega'$. In order to bound this number, we fix $\omega^0\in\calE$, $|S^0|\le \delta n$ and $\pi^0$ such that $\omega'=T(\omega^0,\pi^0,S^0)$.
By locality again, for every maximal path $\pi'$ in $\omega'$, there exists a maximal path in $\omega^0$ coinciding with it, except at some sites $y$ of $S^0$, where $\pi'$ goes through $y^*$ instead of $y$. Since $|S^0|\le \delta n$ and since $|\Pim(\omega^0)|\le 2^{\delta n}$ (since $\omega^0\in\calE$), we deduce that
\begin{equation}\label{eq:aa}|\Pim(\omega')|\le 2^{\delta n}|\Pim(\omega^0)|\le 4^{\delta n}.\end{equation}

Second, fix $\pi^1$ such that there exist $\omega\in\calE$ and $S$ with $|S|\le \delta n$ satisfying that $(\pi^1,S)\in \calS(\omega)$ and $\omega'=T(\omega,\pi^1,S)$. We wish to bound by $2^{3\delta n}$ the number of possible choices for such $S$. Item 1 of locality implies that the (no more than $\delta n$) elements of any possible choice  of $S$ are bifurcations of $\pi^1$ in $\omega'$. It is therefore sufficient to show that  there are at most $3\delta n$ bifurcations of $\pi^1$ in $\omega'$. 
In order to prove this statement, pick $\omega^1\in\calE$ and $S^1$ such that $\omega'=T(\omega^1,\pi^1,S^1)$. There are no more than $2\delta n$ bifurcations of $\pi^1$ in $\omega^1$ since otherwise $|\Pim(\omega^1)|\ge 2^{\delta n}$. Furthermore, there are at most $\delta n$ bifurcations of $\pi^1$ in $\omega'$ which are not bifurcations in $\omega^1$ since they are all included in $S^1$ (by Item 1 of locality).

Overall, the first paragraph implies that there are at most $4^{\delta n}$ choices for $\pi$. The second paragraph implies that there are at most $2^{3\delta n}$ choices for $S$ once $\pi$ is fixed. The claim follows readily.

%
%
\end{proof}

\section{Proof of Theorem~\ref{thm:main}}
 
Recall that, by Corollary~\ref{cor:many turns}, for every $\beta>0$, there exists $\kappa=\kappa(\beta)>0$ such that for every $x\in C_\beta$ with $\|x\|$ large enough,
\begin{equation}\label{eq:many turn}
\bbP(\exists \pi\in \vec\Pi^{(x)}_{\rm max}:|\mathbf{turn}_\pi|< \kappa \|x\|)\le \exp(-\kappa \|x\|).
\end{equation}
The proof of the theorem relies on the following three lemmas. In order to highlight the global strategy of the proof, we postpone the proofs of the lemmas.

The first lemma treats the case in which maximal paths are composed of sites with weight 1 only. This case corresponds to the supercritical oriented percolation case.   \begin{lemma}\label{prop:perco}
For every $\beta>0$, there exists $\delta_1>0$ such that for every $x\in C_\beta$ with $\|x\|$ large enough,
\begin{equation}\label{eq:kjh}\bbP(W_{\rm max}^{(x)}=\|x\|+1 \text{ and }|\vec\Pi^{(x)}_{\rm max}|\le 2^{\delta_1 \|x\|})\le \exp(-\delta_1 \|x\|).\end{equation}
\end{lemma}
The proof relies on Lemma~\ref{principle}. The transformation $T$ (see Fig.~\ref{fig:2}) consists in taking a subset $S$ of cardinality $\delta n$ of well-separated turns of $\pi$, and, for every $y\in S$, changing $\omega_{y^*}$ to 1 and $\omega_z$ to 0 for each $z\in{\bf shield}_\pi(y)$. This transformation will be proved to be local, a fact which will allow us to show that the left-hand side of \eqref{eq:kjh} must be small.

The second lemma deals with the case in which maximal paths have weight very close to $\|x\|+1$, meaning that the average weight of sites on them is very close to $1$.
\begin{lemma}\label{prop:intermediate}
There exists $\delta_2>0$ such that for every $x\in \bbZ_+^d$ with $\|x\|$ large enough,
$$\bbP\big(W_{\rm max}^{(x)}>(1-\delta_2)\|x\| \text{ and }|\vec\Pi^{(x)}_{\rm max}|\le 2^{\delta_2 \|x\|}\big)\le \exp(-\delta_2 \|x\|).$$
\end{lemma}

The proof relies again on Lemma~\ref{principle} combined with Lemma~\ref{prop:perco}. This time, the transformation simply consists in turning all the weights of $\pi$ to 1.

We are now getting closer to the end of the proof since we only need to tackle the case in which $W_{\rm max}^{(x)}\le (1-\delta_2) \|x\|$. In order to do this, we will need a slightly stronger notion of turn.

\begin{definition}Let $R>0$ be an integer. Consider a path $\pi$ from $0$ to $x$. A turn $y=\pi_{k_0}\in \mathbf{turn}_\pi$ is {\em $R$-good} if $R<k_0<\|x\|-R$ and
 $$\sum_{k=k_0-R}^{k_0-1}\omega_{\pi_k}< R-1\qquad\text{and}\qquad\sum_{k=k_0+1}^{k_0+R}\omega_{\pi_k}< R-1.$$

\end{definition}
In words, an $R$-good turn $y$ is a turn such that the path does not gather too high weight in the $R$ steps preceding and following it. Let $\mathbf{gturn}_\pi$ be the set of $R$-good turns of $\pi$.

The third lemma shows that a good proportion of turns are in fact $R$-good. The proof relies heavily on Lemma~\ref{prop:intermediate}. 
\begin{lemma}\label{lem:manyturns}
For every $\beta>0$, there exist $R>0$ and $\delta_3>0$ such that for every $x\in C_\beta$ with $\|x\|$ large enough,
$$\bbP\big(|\vec\Pi^{(x)}_{\rm max}|\le 2^{\delta_3\|x\|}\text{ and  } \exists\pi\in \vec\Pi^{(x)}_{\rm max}:|\mathbf{gturn}_\pi|<\delta_3 \|x\|\big)\le \exp(-\delta_3\|x\|).$$
\end{lemma}
Before proving all these lemmas, let us briefly mention how we will conclude the proof (the proof is postponed to the end of the section). The end of the proof is also relying on Lemma~\ref{principle}. In this case, the transformation consists in picking a subset $S$ of cardinality $\delta n$ of the $R$-good turns of $\pi$, and for each $z$ within distance $R$ of $y\in S$, increasing $\omega_z$  to $1$ if $z\in \pi\cup\{y^*\}$, and decreasing $\omega_z$ to $0$ otherwise (see Fig.~\ref{fig:4}). This creates a bifurcation at $y$, and the choice of $R$ guarantees that the transformation is local. 

We now focus on the proofs of the different statements mentioned above.

\begin{figure}
\centerline{\includegraphics[width=0.15\textwidth]{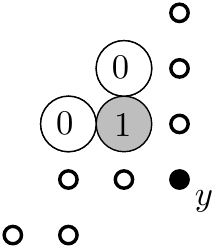}}
\caption{\label{fig:2}
The transformation used in Lemma~\ref{prop:perco} zoomed near a site $y\in S$: the weight of $y^*$ is increased to 1, and the ones of ${\bf shield}_\pi(y)$ set to $0$. The small circles correspond to points of $\pi$ whose weight remains unchanged.}
\end{figure}

\begin{proof}[Lemma~\ref{prop:perco}]
Fix $\beta>0$. Let $\kappa$ be given by \eqref{eq:many turn} and fix $0<\delta\ll \kappa$ to be chosen later. Consider $x\in C_\beta$ with $\|x\|$ large enough. To simplify the notation, we write $n:=\|x\|$. We apply Lemma~\ref{principle} to 
\begin{enumerate}
\item  $\calE=\{W_{\rm max}^{(x)}= n+1\}\cap\{|\vec\Pi^{(x)}_{\rm max}|\le 2^{\delta n}\}\cap\{\forall\pi \in \vec\Pi^{(x)}_{\rm max},|\mathbf{turn}_\pi|\ge \kappa n\},$
\item $\calE'$ is the full probability space, 
\item $\mathcal S(\omega)$ is the set of $(\pi,S)$ for which $\pi\in\Pim(\omega)$ and $S\subset \mathbf{turn}_{\pi}$ is such that $|S|=\delta n$ and any two sites of $S$ are at $\|\cdot\|-$distance at least $2$ of each other,
\item $\displaystyle T(\omega,\pi,S)_z:=\begin{cases}1&\text{ if }z=y^*\text{ for some }y\in S,\\
0&\text{ if }z\in\mathbf{shield}(y)\text{ for some }y\in S,\\
\omega_z&\text{ otherwise}.\end{cases}$
\end{enumerate}
The first two cases in the definition of $T(\omega,\pi,S)$ are never in conflict. Indeed, for every $y\in S$, $\|y^*\|=\|y\|$ and $\|z\|=\|y\|\pm1$ if $z\in\mathbf{shield}(y)$. Furthermore, the $\ell^1$ distance between two distinct elements of $S$ is larger than or equal to 2. 

Equation \eqref{eq:principle} implies that for $n$ large enough, 
\begin{equation}\label{eq:1a}
\bbP(\calE)\le \frac{2^{5\delta n}}{p^{(2d-1)\delta n}\binom{\kappa n-\delta n}{\delta n}}\le \frac{2^{5\delta n}}{p^{(2d-1)\delta n}(\tfrac{\kappa-\delta}{\delta})^{\delta n}},\end{equation}
where we used that
\begin{itemize}
\item $\bbP(\calE')\le 1$, 
\item $|\calS(\omega)|\ge \binom{\kappa n-\delta n}{\delta n}$ for every $\omega\in\calE$ (the term accounts for the number of possible $S$, which is larger than $\binom{\kappa n-\delta n}{\delta n}$ since $\pi\in\vec\Pi_{\rm max}^{(x)}(\omega)$ has more than $\kappa n$ turns. Note that we did not use the fact that there are a priori $|\vec\Pi_{\rm max}^{(x)}(\omega)|$ choices for $\pi$, since this number may be as small as 1), 
\item  \eqref{eq:condition} is satisfied with $\varepsilon=p^{(2d-1)\delta n}$ (since $\delta n$ sites are changed to 1, and at most $(2d-2)\delta n$ to $0$),
\item $|\calT(\omega')|\le 2^{5\delta n}$ for every $\omega'\in\calE'$, as shown using Lemma~\ref{lem:local} and the following claim:
\end{itemize}
\noindent{\bf Claim 1} {\em The transformation $T$ is local.}\medbreak
\noindent{\em Proof of Claim 1.} 
Let $\omega'=T(\omega,\pi,S)\in \calE'$. First, $\pi$ is obviously in $\Pim(\omega')$ and the set of bifurcations of $\pi$ in $\omega$ is the union of $S$ and the set of bifurcations of $\pi$ in $\omega$.

Second, pick a maximal path $\pi'$ in $\omega'$ and $z_0,z_1\in\pi$ with $\pi\cap\pi'[z_0,z_1]=\{z_0,z_1\}$ and $W_{\pi'(z_0,z_1)}<W_{\pi(z_0,z_1)}$. We wish to show that $\pi'(z_0,z_1)=\{y^*\}$ for some $y\in S$, which will conclude the proof. 

Since $W_{\pi'(z_0,z_1)}<W_{\pi(z_0,z_1)}$, $\pi'(z_0,z_1)$ must contain a site $u$ with $\omega_u<1$. Since $\pi'$ is maximal for $\omega'$, it is made of sites with weight 1 only, so that $\omega'_u=1$. The constraint $\omega_u<\omega'_u=1$ implies that $u=y^*$ for some $y\in S$. Now, none of the sites of  ${\bf shield}_y$ can be in $\pi'(z_0,z_1)$ since their weight in $\omega'$ is strictly smaller than 1. This implies that the sites of $\pi'(z_0,z_1)$ before and after $u$ must be in $\pi$, i.e., must be equal to $z_0$ and $z_1$ respectively. In conclusion, $\pi'(z_0,z_1)=\{y^*\}$. 
$\square$
\medbreak
We are now ready to finish the proof. 
Provided that $\delta\ll \kappa$, \eqref{eq:1a} and \eqref{eq:many turn} give that for $n$ large enough,
\begin{align*}\bbP\big(W_{\rm max}^{(x)}=n+1\text{ and }|\vec\Pi^{(x)}_{\rm max}|\le 2^{\delta n}\big)&\le \bbP(\mathcal E)+\bbP\big(\exists\pi \in \vec\Pi^{(x)}_{\rm max},|\mathbf{turn}_\pi|< \kappa n\big)\\
&\le \exp(-\delta n)+\exp(-\kappa n),
\end{align*}
which implies the statement readily by choosing $\delta_1>0$ small enough.
 \end{proof}
 
\begin{proof}[Lemma~\ref{prop:intermediate}]Consider $c>0$ such that the second largest element of $\Theta$ is smaller than $1-c$, and recall that the probability that $\omega_x=1$ is smaller or equal to $1-p$. Let $\delta\ll\beta\ll1$ be two positive numbers to be fixed later. Consider $x\in\bbZ^d_+$ and, as before, set $n:=\|x\|$. 
Notice that a path $\pi$ (from 0 to $x$) with $W_\pi\ge (1-\delta) n$ contains at most $\delta n/c+1$ sites with weights strictly smaller than 1 (and, therefore, at least $(1-\delta/c)n$ sites with weights equal to 1).

Let us first assume that $x\notin C_\beta$. In such case, there are at most $(d-1)^{\beta n+1}\binom {n+1}{\beta n+1}$ oriented paths from 0 to $x$. Since the probability, for each such path, of having at least $(1-\delta/c)n$ sites with weights equal to 1 is smaller than $\binom{n+1}{(1-\delta/c)n}(1-p)^{(1-\delta/c)n}$, we deduce that for $\|x\|$ large enough,
\begin{equation}\label{eq:lj}\bbP\big(W_{\rm max}^{(x)}>(1-\delta)n\big)\le (d-1)^{\beta n+1}\binom {n+1}{\beta n+1}\binom{n+1}{(1-\delta/c)n}(1-p)^{(1-\delta/c)n}\le \exp(-\delta n),\end{equation}
provided that $\beta=\beta(p,c)>0$ and $\delta=\delta(\beta,p,c)$ are small enough. 

In the second part of the proof, consider the constant $\beta$ defined in the last paragraph. Let us assume that $x\in C_\beta$ with $\|x\|$ large enough.  Let $\delta_1=\delta_1(\beta)>0$ be as in Lemma~\ref{prop:perco} and assume that $\delta\ll\delta_1$.
We apply Lemma~\ref{principle} to 
\begin{enumerate}
\item $\mathcal E=\{W_{\rm max}^{(x)}>(1-\delta)n\}\cap\{|\vec\Pi^{(x)}_{\rm max}|\le 2^{\delta n}\}$, 
\item $\calE'=\{W_{\rm max}^{(x)}=n+1\text{ and }|\vec\Pi^{(x)}_{\rm max}|\le 2^{\delta n}\}$,
\item $\calS(\omega)=\{(\pi,S):\pi\in\Pim(\omega)\text{ and }S=\{y\in\pi:\omega_y<1\}\}$, 
\item $\displaystyle T(\omega,\pi,S)_z=\begin{cases} 1&\text{ if }z\in S,\\
\omega_z&\text{ otherwise}.\end{cases}$\end{enumerate} 
Then, \eqref{eq:principle} implies that 
\begin{equation}\label{eq:1b}\bbP(\calE)\le \frac{2^{\delta n}\,\binom{n+1}{\delta n/c+1}}{p^{\delta n/c+1}}\exp(-\delta_1 n),\end{equation}
where we used that
\begin{itemize}
\item $\bbP(\calE')\le\exp(-\delta_1 n)
$
by Lemma~\ref{prop:perco},
\item $|\calS(\omega)
|\ge1$, 
\item \eqref{eq:condition} is satisfied with $\varepsilon=p^{\delta n/c+1}$ since the weights of at most $\delta n/c+1$ sites are changed to 1,
\item  
$
|\calT(\omega')|\le 2^{\delta n}\binom{n+1}{\delta n/c+1}
$ for every $\omega'\in\calE'$
(we used the fact that maximal paths in $\omega'=T(\omega,\pi,S)$ are maximal paths in $\omega$, so that $|\vec\Pi^{(x)}_{\rm max}(\omega')|\le 2^{\delta n}$, and that $S$ is a subset of $\pi$ with $|S|<\delta n/c+1$).
 \end{itemize}
Then, \eqref{eq:1b}  implies that for $n$ large enough,
\begin{align*}\bbP\big(W_{\rm max}^{(x)}>(1-\delta)n\text{ and }|\vec\Pi^{(x)}_{\rm max}|\le 2^{\delta n}\big)
&\le \exp(-\delta n)\end{align*}
by choosing $\delta=\delta(\delta_1,\beta)>0$ small enough. 

Combining this statement with \eqref{eq:lj} concludes the proof.
\end{proof} 


\begin{remark}Note that in the last proof, the lower bound $|\mathcal S(\omega)|\ge1$ does not counterbalance the upper bound on $\calT(\omega')$ on its own, so that the bound $\bbP(\calE')\le \exp(-\delta_1 n)$ given by the previous lemma is necessary. \end{remark}

\begin{proof}[Lemma~\ref{lem:manyturns}]
Fix $\beta>0$. Let $\delta_2$ be given by Lemma~\ref{prop:intermediate} and $\kappa$ by \eqref{eq:many turn}. Consider $x\in C_\beta$ with $\|x\|$ large enough. As before, we set $n:=\|x\|$. Consider $r>0$ to be fixed later and set $m=\lfloor n/r\rfloor$. For every $j\ge1$, introduce 
$$T_{j}:=\{y\in\bbZ^d_+:(j-1)r\le \|y\|<jr\}.$$
Let $\calJ$ be the set of $J\subset\{1,\dots,m\}$ with $|J|>\tfrac\kappa4 m$. Let $\calX$ be the set of $x_0,x_1,\dots,x_m$ such that $x_0=0$ and for every $i\in\{1,\dots,m\}$, there exists an oriented path of length $r+1$ from $x_{i-1}$ to $x_i$ (recall that the length of a path from 0 to $y$ is $\|y\|+1$).  
\medbreak
Let $\calE$ be the event that there exist $(x_0,\dots,x_m)\in \calX$ and a set $J\in \calJ$  such that for every $j\in J$, \begin{itemize}
\item the number of maximal paths from $x_{j-1}$ to $x_j$ is smaller than $2^{\delta_2 r}$,
\item the maximum $W^{x_{j-1},x_j}$ of the $W_{\pi\cap T_j}$ over oriented paths $\pi$ from $x_{j-1}$ to $x_j$ is greater than or equal to $r-1$.\end{itemize}
For each $j$, the two items above depend on weights in $T_j$ only. Since the $T_j$ are disjoint, the union bound and independence imply that 
\begin{align*}\bbP(\calE)&\le \sum_{\substack{(x_0,\dots,x_{m})\in\calX\\ J\in\calJ}}\ \prod_{j\in J}\bbP\big(W_{\rm max}^{(x_{j}-x_{j-1})}\ge r-1 \text{ and }|\vec\Pi_{\rm max}^{(x_{j}-x_{j-1})
}|\le 2^{\delta_2r}\big),\end{align*}
where we used the invariance under translations and the fact that $W^{x_{j-1},x_j}$ is smaller than or equal to the maximal weight of oriented paths from $x_{j-1}$ to $x_j$.

Assume now that $r$ is so large that $r-1>(1-\delta_2)(r+1)$. Lemma~\ref{prop:intermediate} implies that 
$$\bbP(\calE)\le \sum_{\substack{(x_0,\dots,x_{m})\in\calX\\ J\in\calJ}}\ \prod_{j\in J}e^{-\delta_2 r}\le r^{(d-1)m }\cdot2^me^{-\delta_2 r\tfrac\kappa4 m}.$$
In the second inequality, we bounded the cardinality of $\calJ$ by $2^m$, and the one of $\calX$ by $r^{(d-1)m }$ (there are at most $r^{d-1}$ choices for each $x_{i}$ since $x_i$ should be reachable from $x_{i-1}$ using a directed path of length $r+1$). 

Provided that $r$ is large enough, we find that $2r^{d-1}e^{-\delta_2 r\kappa/4}<1$, so that there exists $\delta>0$ such that for every $n$ large enough,
\begin{align}\label{eq:2d}\bbP(\calE)\le \exp(-\delta n).\end{align}

Now, fix $\delta=\delta(\delta_2)>0$ so small that  $\kappa> 4(1+2/\delta_2)\delta$. Define the event
$$\calG:=\calE^c\cap\{|\vec\Pi^{(x)}_{\rm max}|\le 2^{\delta n}\}\cap\{\forall \pi\in\vec\Pi_{\rm max}^{(x)}:|{\bf turn}_\pi|\ge \kappa n\}.$$
We wish to show that on the event $\calG$, maximal paths possess many $R$-good turns. Consider $\omega\in\calG$ and $\pi$ a maximal path in $\omega$ from 0 to $x$. Then,
\begin{itemize}
\item Since $\calG\subset\calE^c$, there are at most $\tfrac\kappa2m$ slabs $T_j$ with either $T_{j+1}$ or $T_{j-1}$ satisfying that $W_{\pi\cap T_{j\pm1}}\ge r-1$ and the number of maximal paths from $\pi_{(j\pm1-1)r}$ to $\pi_{(j\pm1)r}$ is smaller than or equal to $2^{\delta_2r}$, and therefore at most $\tfrac\kappa2n$ turns within them. 
\item Since $\calG\subset \{|\vec\Pi^{(x)}_{\rm max}|\le 2^{\delta n}\}$, there are at most $(2\delta/\delta_2)(m+1)$ slabs $T_j$ with either $T_{j+1}$ or $T_{j-1}$ such that the number of maximal paths from $\pi_{(j\pm 1-1)r}$ to $\pi_{(j\pm1)r}$ is larger than $2^{\delta_2r}$, and therefore at most $(2\delta/\delta_2)(n+r)$ turns within them.
\item There are at most $4r$ turns that are at a distance smaller than or equal to $2r$ from the endpoints.
\end{itemize}
The choice of $\delta$ implies that for $n$ large enough, at least $\delta n$ turns are in none of the previous cases since
$$(\kappa-\tfrac\kappa2-2\delta/\delta_2)n-(4+2\delta/\delta_2)r>\delta n.$$In particular, these turns are necessarily $R$-good for $R=2r+1$. In conclusion, \eqref{eq:2d} and \eqref{eq:many turn} imply
\begin{align*}\bbP\big(|\vec\Pi^{(x)}_{\rm max}|\le 2^{\delta n}\text{ and  } \exists\pi\in \vec\Pi^{(x)}_{\rm max}:|\mathbf{gturn}_\pi|<\delta n\big)&\le\bbP(\calE)+\bbP(\exists \pi\in\vec\Pi_{\rm max}^{(x)}:|{\bf turn}_\pi|< \kappa n)\\
&\le \exp(-\delta n)+\exp(-\kappa n),\end{align*}
 which concludes the proof by setting $\delta_3>0$ small enough.
\end{proof}

\begin{remark}
The previous lemma requires to study point-to-point maximal paths, which was one of our motivation to work in this context rather than with point-to-anywhere maximal paths.
\end{remark}
We are now ready to prove Theorem~\ref{thm:main}. 

\begin{figure}
\centerline{\includegraphics[width=0.25\textwidth]{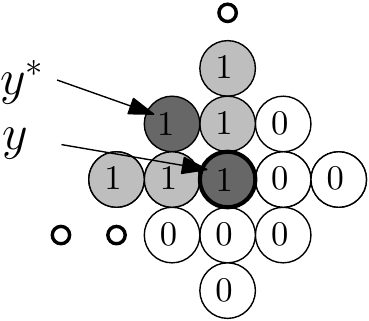}}
\caption{\label{fig:4}
The transformation used in the proof of Theorem~\ref{thm:main}: within the box of size $R$ around $y$, the weights on $\pi\cup\{y^*\}$ are turned to $1$, the others to $0$.}
\end{figure}

\begin{proof}[Theorem~\ref{thm:main}]
Fix $\beta>0$. 
Let $\delta\ll \delta_3$ be fixed later. As usual, set $n:=\|x\|$. Fix $R$ and $\delta_3>0$ given by the previous lemma. 

We wish to apply Lemma~\ref{principle}
 to 
\begin{enumerate}
\item $\calE=\{|\vec\Pi^{(x)}_{\rm max}|\le 2^{\delta n}\}\cap\{\forall \pi\in \vec\Pi^{(x)}_{\rm max}:|\mathbf{gturn}_\pi|\ge \textstyle\delta_3n\},$
\item $\calE'$ is the full probability space,
\item $\mathcal S(\omega)$ is the set of $(\pi,S)$ for which $\pi\in\Pim(\omega)$ and $S\subset \mathbf{gturn}_{\pi}$ is such that $|S|=\delta n$ and any two sites of $S$ are at a distance at least $2R+1$ of each other,
\item $\displaystyle T(\omega,\pi,S)_z=\begin{cases}
0&\text{ if for some }y\in S,\ z\notin\pi\cup\{y^*\}\text{ and }\|z-y\|\le R,\\
1&\text{ if for some }y\in S,\ z\in \pi\cup\{y^*\}\text{ and }\|z-y\|\le R,\\
\omega_z&\text{ if for all }y\in S,\ \|z-y\|> R.\end{cases}$
 \end{enumerate}
(See Fig.~\ref{fig:4}.) Again, the first two  cases of the definition of $T(\omega,\pi,S)$ are never in conflict. Then, \eqref{eq:principle} shows that 
\begin{equation}\label{eq:aab}\bbP(\calE)\le \frac{2^{5\delta  n}}{p^{(2R+1)
^d\delta n}\binom{\delta_3 n/(2R+1)}{\delta n}}\le \frac{2^{5\delta  n}}{p^{(2R+1)
^d\delta n}(\tfrac{\delta_3/(2R+1)}{\delta})^{\delta n}} ,\end{equation}
where we used that
\begin{itemize}
\item $\bbP(\calE')\le 1$,
\item  $|\calS(\omega)|\ge \binom{\delta_3 n/(2R+1)}{\delta n}$ for every $\omega\in \calE$ (this lower bound is obtained by first determining a subset of $\mathbf{gturn}_\pi$ containing $\delta_3n/(2R+1)$ turns at a distance at least $2R$ of each other, and then choosing one of its $\binom{\delta_3 n/(2R+1)}{\delta n}$ subsets of cardinality $\delta  n$),
\item \eqref{eq:condition} is satisfied with $\varepsilon= p^{(2R+1)^d\delta n}$ (we changed at most $(2R+1)^d\delta n$ weights),
\item $|\calT(\omega')|\le 2^{5\delta  n}$ for every $\omega'\in \calE'$, as shown using Lemma~\ref{lem:local} and the following claim:
\end{itemize}
\medbreak 
\noindent{\bf Claim 2} {\em The transformation $T$ is local.}
\medbreak
Before proving the claim, let us finish the proof of the theorem.
Provided that $\delta\ll\delta_3$, Lemma~\ref{lem:manyturns} and \eqref{eq:aab} imply that for $n$ large enough,
\begin{align*}\bbP(|\vec\Pi^{(x)}_{\rm max}|\le 2^{\delta  n})&\le \bbP(\calE)+\bbP(|\vec\Pi^{(x)}_{\rm max}|\le 2^{\delta_3n}\text{ and  } \exists\pi\in \vec\Pi^{(x)}_{\rm max}:|\mathbf{gturn}_\pi|<\delta_3 n)\\
&\le \exp(-\delta n)+\exp(-\delta_3 n),\end{align*}
which implies the theorem readily by fixing $\delta$ small enough. We therefore simply need to prove the claim to conclude the proof.
\medbreak
\noindent{\em Proof of Claim 2.} 
Let $\omega'=T(\omega,\pi,S)\in \calE'$. 
Pick a maximal path $\pi'$ in $\omega'$. Consider $z_0,z_1\in\pi$ with $\pi\cap\pi'[z_0,z_1]=\{z_0,z_1\}$. We wish to show that $W_{\pi'(z_0,z_1)}'\le W_{\pi(z_0,z_1)}'$ and that if $W_{\pi'(z_0,z_1)}<W_{\pi(z_0,z_1)}$, then  $\pi'(z_0,z_1)=\{y^*\}$ for some $y\in S$.

This will immediately imply that $T$ is local. Indeed, the first item of locality follows from the fact that $\pi$ is maximal in $\omega'$ and that a turn $z\in\pi$ which is a bifurcation of $\pi'$ but not a bifurcation of $\pi$ must be in $S$. To show this last fact, note that the ball (for the $\ell^1$ norm) of radius $R$ around a vertex $y\in S$ contains $z$ if and only if it contains $z^*$, so that the weight of a turn $z\in \pi$ is redefined if and only if the one of $z^*$ is, and that in such case $\omega'_z\ne\omega'_{z^*}$ if $z\notin S$. The second item of locality follows readily from the last claim of the last paragraph.

In order to prove that claim, let us consider two cases. In the first one, there does not exist $y\in S$ such that $y^*\in\pi'(z_0,z_1)$. The weights in $\omega'$ of sites of $\pi'(z_0,z_1)$ are then smaller or equal to those in $\omega$, so that
$$ W'_{\pi'(z_0,z_1)}\le W_{\pi'(z_0,z_1)}\le W_{\pi(z_0,z_1)}\le W'_{\pi(z_0,z_1)}.$$
The second inequality is due to the fact that $\pi$ is maximal in $\omega$. Note that if $W_{\pi'(z_0,z_1)}<W_{\pi(z_0,z_1)}$, the middle inequality is strict (a fact which contradicts the maximality of $\pi'$ in $\omega'$) so that we cannot be in this case.

In the second case, there exists $y\in S$ such that $y^*\in\pi'(z_0,z_1)$. We wish to show that $\pi'(z_0,z_1)=\{y^*\}$. Note that since $\omega'_y=\omega'_{y^*}$, this will give that $W'_{\pi'(z_0,z_1)}=W'_{\pi(z_0,z_1)}$. 

If $y_0$ is the first site $y\in S$ such that $y^*\in\pi'(z_0,z_1)$, we claim that $z_0$ must be the site preceding it in $\pi$ since otherwise $\pi'$ cannot be maximal in $\omega'$. Indeed, in such case, replacing $\pi'(z_0,y_0^*)$ by $\pi(z_0,y_0)$ in $\pi'$ would strictly increase the weight of the path in $\omega'$, as we now justify by dividing into two cases:\begin{itemize}
\item If $\|z_0-y_0\|\le R$ and $z_0$ is not preceding $y_0$ in $\pi$, then $\pi'(z_0,y_0^*)$ is non-empty and is composed of sites of weight $0$ in $\omega'$, while the path $\pi(z_0,y_0)$ can be substituted for $\pi'(z_0,y_0^*)$ in $\pi'$ and is composed only of sites of weight $1$ in $\omega'$, which contradicts the maximality of $\pi'$ in $\omega'$.
\item If $\|z_0-y_0\|>R$, set $t_0$ for the site of $\pi$ after $y_0$.  Then,
\begin{align*}W'_{\pi'(z_0,y_0^*)}&\le W_{\pi'(z_0,y_0^*)}\\
&\le W_{\pi(z_0,t_0)}-\omega_{y_0^*}\\
&=W_{\pi(z_0,y_0)}+\omega_{y_0}-\omega_{y_0^*}\\
&< W'_{\pi(z_0,y_0)},\end{align*}
where the first inequality is due to the fact that $\pi'(z_0,y_0^*)$ does not intersect $\{y^*:y\in S\}$, the second to the fact that $\pi$ is maximal in $\omega$ between $z_0$ and $t_0$, and the third to the fact that $y_0$ is a $R$-good turn so that $W_{\pi(z_0,y_0)}+1<W'_{\pi(z_0,y_0)}$. The last claim can be obtained from the following sequence of inequalities (the first one is due to the fact that $y_0$ is $R$-good)
\begin{equation}\label{eq:aaga}\sum_{\substack{z\in \pi(z_0,y_0)\\ \|z-y_0\|\le R}}\omega_z<R-1=\sum_{\substack{z\in \pi(z_0,y_0)\\ \|z-y_0\|\le R}}\omega'_z-1.\end{equation}
\end{itemize} 

The same reasoning implies that $z_1$ must be the site of $\pi$ following the last site $y_1\in S$ such that $y^*\in\pi(z_0,z_1)$. It only remains to exclude the case where $\pi'(z_0,z_1)$ contains more than one site of $\{y^*:y\in S\}$. In such case, pick two consecutive such sites $y_2$ and $y_3$.  Then, basically the same reasoning applies. Indeed, in this case as well $\pi'$ cannot be maximal since
\begin{align*}W'_{\pi'(y_2^*,y_3^*)}&\le W_{\pi'(y_2^*,y_3^*)}=W_{\pi'[y_2^*,y_3^*]}-\omega_{y_2^*}-\omega_{y_3^*}\\
&\le W_{\pi[y_2,y_3]}-\omega_{y_2^*}-\omega_{y_3^*}=W_{\pi(y_2,y_3)}+\omega_{y_2}+\omega_{y_3}-\omega_{y_2^*}-\omega_{y_3^*}\\
& < W'_{\pi(y_2,y_3)}.\end{align*}
The second line uses that $\pi$ is maximal in $\omega$ between the vertex before $y_2$ and the one after $y_3$. The last line follows from $W_{\pi(y_2,y_3)}+2<W'_{\pi(y_2,y_3)}$, an inequality which is obtained in the same way as \eqref{eq:aaga} using the fact that $y_2$ and $y_3$ are $R$-good turns.
\end{proof}
%

\paragraph{Acknowledgements} The first author is partly supported by the IDEX chair of Paris-Saclay, as well as the NCCR Swissmap. The third author was supported by the NSF grant DMS-1600239. The last author was supported by 
CNPq Bolsa de Produtividade grant.
 We thank Ryoki Fukushima and Shuta Nakajima for useful comments. \providecommand{\bysame}{\leavevmode\hbox to3em{\hrulefill}\thinspace}
\providecommand{\MR}{\relax\ifhmode\unskip\space\fi MR }
\providecommand{\MRhref}[2]{%
  \href{http://www.ams.org/mathscinet-getitem?mr=#1}{#2}
}
\providecommand{\href}[2]{#2}

%
\end{document}